\documentclass[final,11pt]{amsart}

\usepackage{showkeys}

\usepackage{amssymb}

\def\T{{\mathcal{T}}}
\def\chic{{\chi_\mathrm{c}}}

\let\oldBox\Box
\def\Box{\mathbin{\oldBox}}

\usepackage[arrow,matrix]{xy}

\title[Hedetniemi's conjecture and adjoint functors]{Hedetniemi's conjecture and adjoint functors in thin categories}
\author[J. Foniok]{Jan Foniok}
\author[C. Tardif]{Claude Tardif}

\address{Manchester Metropolitan University \\
School of Computing, Mathematics and Digital Technology \\
John Dalton Building \\
Chester Street \\
Manchester M1 5GD \\
England}
\email{j.foniok@mmu.ac.uk}

\address{Royal Military College of Canada \\ 
PO Box 17000 Station ``Forces'' \\
Kingston, Ontario\\ 
Canada, K7K 7B4 } 
\email{Claude.Tardif@rmc.ca}
\urladdr{http://www.rmc.ca/academic/math\underline{\ }cs/tardif/}


\usepackage{amsthm}
\newtheorem{lemma}{Lemma}[section]
\newtheorem{theorem}[lemma]{Theorem}
\newtheorem{corollary}[lemma]{Corollary}

\newtheorem{prop}[lemma]{Proposition}
\newtheorem{conjecture}[lemma]{Conjecture}
\theoremstyle{definition}

\newtheorem{problem}[lemma]{Problem}

\newcommand{\modheq}{/{\leftrightarrow}}
\newfont{\Bb}{msbm10 scaled\magstep1}

\begin{document}


\begin{abstract}
We survey results on Hedetniemi's conjecture which are connected
to adjoint functors in the ``thin'' category of graphs, and expose
the obstacles to extending these results.
\end{abstract}

\date{9 August, 2016}
\maketitle

\section{Introduction and terminology}
Hedetniemi's conjecture states that the chromatic number
of the product of two graphs is equal to the minimum of the chromatic 
numbers of the factors. This conjecture has caught the
attention of some category theorists, who recognize their turf
in its setting of morphisms and products in Cartesian closed 
categories. Yet, the input from category theorists has not gone much beyond 
reformulations of the problem and rediscovery of basic results.

In our opinion there is potential for deeper input.
However, we will argue that the consideration of the
graph-theoretic product in terms of its categorial
definition is perhaps the wrong approach. The most
relevant property of the product seems to be that 
it is preserved by right adjoint functors. Indeed, 
significant results have been proved using this fact.

The purpose of this paper is to expose this point of view
in greater detail. However for this purpose,
we will need to escape the usual category of graphs and homomorphisms and
work instead in thin categories.

\subsection{Graphs, digraphs, homomorphisms and chromatic numbers}
\label{ss:gdhcn}

It will be useful to consider both finite graphs and finite digraphs. 
A {\em digraph}~$G$ consists of the following data:
1.~a finite set~$V(G)$ of {\em vertices};
2.~a finite set~$A(G)$ of {\em arcs};
3.~two maps $t_G,h_G:A(G)\to V(G)$ that determine the {\em tail} and the {\em head}
of each arc.
For $a\in A(G)$, we say that it is an arc from $t_G(a)$ to $h_G(a)$
and write $a:t_G(a)\to h_G(a)$.

A digraph is {\em simple} if for any $u,v\in V(G)$ there is at most one arc from $u$ to~$v$.
For simple digraphs it is common for an arc $u\to v$ to be called $(u,v)$,
and so $A(G)$ is usually viewed as a binary relation on~$V(G)$.
When $A(G)$ is symmetric, $G$ is called a {\em graph};
an {\em edge} of $G$ is a pair $\{u, v\}$ with $(u,v), (v,u) \in A(G)$.

A {\em homomorphism}~$f$ of a digraph~$G$ to a digraph~$H$ is a pair of maps $(f_0,f_1)$,
$f_0:V(G)\to V(H)$, $f_1:A(G)\to A(H)$,
such that both of the following diagrams commute:
\[
  \xymatrix{
    A(G) \ar[r]^{f_1} \ar[d]^{t_G}      & A(H) \ar[d]^{t_H} \\
    V(G) \ar[r]^{f_0}                   & V(H) 
  }
  \qquad
  \xymatrix{
    A(G) \ar[r]^{f_1} \ar[d]^{h_G}      & A(H) \ar[d]^{h_H} \\
    V(G) \ar[r]^{f_0}                   & V(H) 
  }
\]
Homomorphisms are composed component-wise.
Where $H$ is simple, $f$~is usually identified with~$f_0$ because
$f_1$~is uniquely determined by it.

The {\em product} $G\times H$ of two digraphs is given by
\begin{align*}
  & V(G\times H) = V(G) \times V(H), \\
  & A(G\times H) = A(G) \times A(H), \\
  & t_{G\times H} = (t_G, t_H) \text{ and } h_{G\times H} = (h_G, h_H).
\end{align*}
This construction, along with the obvious projections,
is the product in the category of digraphs and their homomorphisms.
Note that symmetry is preserved by the product, that is, the product of two graphs 
is a graph. 

The {\em complete graph} $K_n$ on $n$ vertices is the (simple) graph defined by
\begin{eqnarray*}
V(K_n) & = & \{0, 1, \ldots, n-1\}, \\
A(K_n) & = & \{ (i,j) : i \neq j \}.
\end{eqnarray*}
An {\em $n$-colouring} of a digraph $G$ is a homomorphism of $G$ to $K_n$.
The {\em chromatic number} $\chi(G)$ of $G$ is the least $n$ such that
$G$ admits an $n$-colouring. Note that a digraph $G$ can have {\em loops}, that is,
arcs of the form $u\to u$. In this case, $G$ has no $n$-colourings for any $n$, and 
we define $\chi(G)=\infty$.

\subsection{Hedetniemi's conjecture and related questions}

Hedetniemi's conjecture states the following:

\begin{conjecture}[\cite{Hed:Con}] \label{hc}
If $G$ and $H$ are graphs, then 
$$\chi(G\times H) = \min \{ \chi(G), \chi(H) \}.$$
\end{conjecture}
The inequality $\chi(G\times H) \leq \min \{ \chi(G), \chi(H) \}$ follows
immediately from definitions, so Hedetniemi's conjecture is equivalent
to the statement that $\chi(G\times H) \geq \min \{ \chi(G), \chi(H) \}$.
Though no counterexamples have been found in fifty years, there is no
categorial reason for Hedetniemi's conjecture to hold, since ``arrows
go the wrong way''. In fact, Hedetniemi's conjecture  fails for digraphs, 
with relatively small examples (see Section~\ref{ag}). 
The following much weaker version of Hedetniemi's conjecture will 
be relevant to our discussion:
\begin{conjecture}[The weak Hedetniemi conjecture] \label{whc}
For every integer $n$, there exists an integer $f(n)$ such that 
if $G$ and $H$ are graphs with chromatic number at least $f(n)$, 
then $$\chi(G\times H) \geq n.$$
\end{conjecture}
We will also be interested in the characterization
of so-called ``multiplicative graphs''. 
A graph $K$ is called {\em multiplicative}
if the existence of a homomorphism of a product $G \times H$
to $K$ implies the existence of a homomorphism of a factor
$G$ or $H$ to $K$. Hence Hedetniemi's conjecture is equivalent
to the statement that the complete graphs are multiplicative.
However, few graphs have yet been shown to be multiplicative.

\subsection{Structure of the paper}

The paper is structured as follows. In Section~\ref{af}
we present the thin categories of graphs and digraphs, and
functors relevant to our discussion. In
Section~\ref{ag} we give an exposition of some of the strange results on 
Conjecture~\ref{whc} obtained using such functors.
In Section~\ref{pf} we see how attempts to improve
on the previous results have led to proving the
multiplicativity of some ``circular complete graphs''; 
we then discuss apparent limitations of the method of 
adjoint functors towards Hedetniemi's conjecture. 
We provide proof of the results which we can prove using
properties of adjoint functors, and we propose open problems 
throughout the paper.

\section{Adjoint functors in the thin categories of graphs and digraphs} \label{af}

\subsection{Thin categories}

In the context of colourings and Hedetniemi's conjecture, the
existence or non-existence of a homomorphism between two graphs or
digraphs seems to be more relevant than the structure of homomorphisms
between them.
Hence we focus our study on thin categories.
A \emph{thin category} is a category in which there is at most one
morphism between any two objects.
The \emph{thin category of digraphs (resp.\ graphs)} is the category whose
objects are digraphs (resp.\ graphs) and where there is a unique morphism
from~$G$ to~$H$ if there exists a homomorphism of~$G$ to~$H$.

If a homomorphism of~$G$ to~$H$ exists,  we write $G \rightarrow H$;
we write $G \not \rightarrow H$ if it does not.
Thus $\rightarrow$ is a binary relation on the
class $\mathcal{D}$ of finite digraphs or on its subclass $\mathcal{G}$
of finite graphs.
More precisely, $\rightarrow$ is a preorder on $\mathcal{G}$ and 
on $\mathcal{D}$, and Hedetniemi's conjecture can be naturally
stated in terms of the quotient order $\mathcal{G}\modheq$.
The relation $\leftrightarrow$ is naturally defined
by $G \leftrightarrow H$ if $G\rightarrow H$ and 
$H \rightarrow G$; $G$ and $H$ are then called {\em homomorphically equivalent}.
Note that $G$ and $H$ are homomorphically equivalent if and only if
they are isomorphic in the thin category.

Both $\mathcal{G}\modheq$ and $\mathcal{D}\modheq$
are distributive lattices with 
$(G\modheq) \wedge (H\modheq) = (G \times H)\modheq$.
(In fact, both $\mathcal{G}\modheq$ and $\mathcal{D}\modheq$
are Heyting algebras because the categories of graphs and digraphs are Cartesian closed.)
An element~$K$ of~$\mathcal{G}$ (resp.\ $\mathcal{D}$) is multiplicative if and only if
$K\modheq$ is meet-irreducible in  $\mathcal{G}\modheq$
(resp.\ $\mathcal{D}\modheq$).
Thus Hedetniemi's conjecture is equivalent to the statement that
for every complete graph~$K_n$, $K_n\modheq$ is meet-irreducible 
in~$\mathcal{G}\modheq$. This reformulation does not
simplify matters since the structure of $\mathcal{G}\modheq$
is complex, and in particular, dense above $K_2$ (see \cite{HelNes:GrH}).

In thin categories, adjunction comes down to an equivalence
between two existential statements.
Specifically, two functors $L, R$ between thin categories ($\mathcal{G}$ or $\mathcal{D}$)
are respectively left and right adjoints of each other if the following 
property holds:
\begin{eqnarray*}
L(G) \rightarrow H & \Leftrightarrow & G \rightarrow R(H).
\end{eqnarray*}
This condition is weaker than adjunction in the usual categories
with all homomorphisms.
Thus all adjoint pairs in these categories (which we present next)
are also adjoint pairs in the thin categories,
but there are other thin adjoints as we shall see in~\ref{ss:racpf}.

\subsection{Pultr template and functors}

The functors relevant to our discussion are connected to 
those introduced by Pultr in \cite{Pul:The-right-adjoints}.
Pultr worked, among others, in the usual categories of graphs and digraphs
(as defined in~\ref{ss:gdhcn}).
He characterized the adjoint functors in these categories by means
of the following construction.
\begin{itemize}
\item[(i)] A {\em Pultr template} is a quadruple $\T = (P,Q,\epsilon_1,\epsilon_2)$
where $P$, $Q$ are digraphs and $\epsilon_1, \epsilon_2$ homomorphisms of $P$ to $Q$.
\item[(ii)] Given a Pultr template $\T = (P,Q,\epsilon_1,\epsilon_2)$,
the \emph{left Pultr functor}~$\Lambda_{\T}$ 
is the following construction:
For a digraph $G$, $\Lambda_{\T}(G)$ contains one copy $P_u$ of $P$
for every vertex $u$ of $G$, and for every arc $a:u\to v$ of~$G$, 
$\Lambda_{\T}(G)$ contains a copy $Q_{a}$ of $Q$ with $\epsilon_1(P)$ identified
with $P_u$ and $\epsilon_2(P)$ identified with~$P_v$.
\item[(iii)] Given a Pultr template $\T = (P,Q,\epsilon_1,\epsilon_2)$
the \emph{central Pultr functor} $\Gamma_{\T}$ 
is the following construction:
For a digraph $H$, the vertices of $\Gamma_{\T}(H)$ are the homomorphisms
$f : P \rightarrow H$, and the edges of $\Gamma_{\T}(H)$ are the homomorphisms
$g : Q \rightarrow H$; such $g$~is an arc from
$f_1 = g \circ \epsilon_1$ to $f_2 = g \circ \epsilon_2$.
\end{itemize}

For any Pultr template $\T$, $\Lambda_{\T}: \mathcal{D} \rightarrow \mathcal{D}$ and 
$\Gamma_{\T}: \mathcal{D} \rightarrow \mathcal{D}$
are respectively left and right adjoints of each other (see~\cite{Pul:The-right-adjoints,FonTar:Adjoint}). 
We use the adjective ``central'' rather
than ``right'' for $\Gamma_{\T}$ because in some important cases, $\Gamma_{\T}$
itself admits a right adjoint in the thin category. 

The right adjoints into~$\mathcal{G}$ need an additional property of the Pultr template:
both $P$ and~$Q$ must be graphs, and moreover,
$Q$~must admit an automorphism~$q$ with $q \circ \epsilon_1 = \epsilon_2$ and 
$q \circ \epsilon_2 = \epsilon_1$.
The existence of such an automorphism
makes the conditions in (ii) and (iii) symmetric, so that $\Lambda_{\T}(G)$
is well defined, and $\Gamma_{\T}(H)$ is a graph rather than a digraph.
(For the adjunction to work in the ``non-thin'' category of graphs and homomorphisms,
we need to allow graphs with multiple edges.)
 
We now give a simple example of a Pultr template, which will be useful
later on.
We let $\T(3) =(K_1,P_3,\epsilon_1,\epsilon_2)$, where $P_3$ is the path with 
three edges and $\epsilon_1, \epsilon_2$ map $K_1$ to the endpoints of~$P_3$.  
Then $\Lambda_{\T(3)}(G)$ is obtained from $G$ by replacing each edge
by a path with three edges, and $\Gamma_{\T(3)}(H)$~is obtained from~$H$ by putting 
edges between vertices joined by a walk of length~$3$ in~$H$. In particular,
if $H$ is the cycle with five vertices, then $\Gamma_{\T(3)}(H)$
is the complete graph~$K_5$. 
The adjunction between $\Lambda_{\T(3)}$ and $\Gamma_{\T(3)}$ then reads as
follows.
\begin{quote}
For any graph $G$, $G$ is 5-colourable if and only if $\Lambda_{\T(3)}(G)$ admits
a homomorphism to the five-cycle. 
\end{quote}
This is a ``complexity reduction'' of the problem of determining whether a graph
is 5-colourable to the problem of determining whether a graph admits a homomorphism
to the five-cycle. Since the former problem is NP-complete, so is the latter.
Using an arsenal of similar reductions, 
Hell and Ne\v{s}et\v{r}il~\cite{HelNes:Dicho} eventually proved that for any
non-bipartite graph $H$ without loops, the problem of determining whether a graph admits a 
homomorphism to $H$ is NP-complete.

\subsection{``Graph products'' and variants on Hedetniemi's conjecture}

Other examples of Pultr templates and functors are connected to the
 various ``products'' used in graph theory. Indeed there are many useful
ways to define an edge set on the Cartesian product $V(G) \times V(H)$
of the vertex sets of simple graphs $G$ and~$H$. We consider the following examples
from \cite{ImrKla:GrP}.
\begin{itemize}
\item The {\em direct product} $G\times H$ is the product
we introduced earlier: it
has edges $\{(u_1,u_2),(v_1,v_2)\}$
such that $\{u_1,v_1\}$ is an edge of $G$  
and $\{u_2,v_2\}$ is an edge of $H$,
\item  The {\em Cartesian product} $G\Box H$ 
has edges $\{(u_1,u_2),(v_1,v_2)\}$
whenever $\{u_1,v_1\}$ is an edge of $G$ and $u_2 = v_2$, or  
$u_1 = v_1$ and $\{u_2,v_2\}$ is an edge of $H$,
\item The {\em lexicographic product} $G\circ H$ 
has edges $\{(u_1,u_2),(v_1,v_2)\}$
such that $\{u_1,v_1\}$ is an edge of $G$ and $u_2, v_2$ are arbitrary, 
or $u_1 = v_1$ and $\{u_2,v_2\}$ is an edge of $H$.
\end{itemize}
For a fixed graph $H$ and $\star \in \{ \times, \Box, \circ\}$, let
$\T(\star, H)$ be the Pultr template 
$(H \star K_1, H \star K_2, \epsilon_1,\epsilon_2)$
where $\epsilon_1(u) = (u,0)$ and $\epsilon_2(u) = (u,1)$.
Then for any graph $G$, $\Lambda_{\T(\star, H)}(G) = G \star H$.
Therefore any product by a fixed graph is a left adjoint, and the equivalence
$$G \star H \rightarrow K \Leftrightarrow G \rightarrow \Gamma_{\T(\star, H)}(K)$$
holds. In particular, $\Gamma_{\T(\times, H)}(K)$ is the exponential graph 
usually denoted $K^H$, which is the exponential object in the category~$\mathcal{G}$.

Exponential graphs connect to Hedetniemi's conjecture as follows: 
We have $G \times H \rightarrow K_n$
if and only if $G \rightarrow  \Gamma_{\T(\times, H)}(K_n) = K_n^H$.
If $\chi(H) \leq n$, any homomorphism $f: H \rightarrow K_n$ corresponds to a
loop in~$K_n^H$, so the condition $G \rightarrow  K_n^H$ is trivially satisfied.
Let $\chi(H) > n$.
Then \[\chi(K_n^H) = \max\{ \chi(G) : G\times H \to K_n \} \]
because  $\chi(G) \leq \chi(K_n^H)$ whenever $G \to K_n^H$,
that is, whenever $G\times H \to K_n$.
Thus, Hedetniemi's conjecture is equivalent to the statement that
\[ \chi(K_n^H) \leq n \text{ whenever } \chi(H) > n. \]

Similar observations hold in the case of the other graph products:
We have $G \Box H \rightarrow K_n$
if and only if $G \rightarrow  \Gamma_{\T(\Box, H)}(K_n)$.
When $\chi(H) > n$, $\Gamma_{\T(\Box, H)}(K_n)$ is empty,
so $G \Box H \rightarrow K_n$ only if $G$ is empty.
It is easy to show that when $\chi(H) \leq n$, 
$\Gamma_{\T(\Box, H)}(K_n)$ is homomorphically equivalent
to~$K_n$, so $G \Box H \rightarrow K_n$ only if $\chi(G) \leq n$.
Therefore, we have 
$$\chi(G\Box H) = \max\{\chi(G), \chi(H)\}.$$
This is the \emph{Sabidussi identity}, which is sometimes viewed
as a companion formula to Hedetniemi's conjecture.

There is no corresponding formula for the lexicographic product.
In general, we have $\chi(G\circ H) = \chi(G \circ K_m)$, 
where $m = \chi(H)$ (see \cite{GellerStahl}). Thus $\chi(G\circ H)$
does not depend on the structure of $H$, just on its chromatic number.
However, the structure of $G$ is relevant. We have  
$G \circ K_m = \Lambda_{\T(\circ, K_m)}(G) \rightarrow K_n$
if and only if $G \rightarrow \Gamma_{\T(\circ, K_m)}(K_n)$,
and the latter is homomorphically equivalent to the
{\em Kneser graph} $K(n,m)$ defined by
\begin{eqnarray*}
V(K(n,m)) & = & \mathcal{P}_m(V(K_n)), \\
A(K(n,m)) & = & \{ (A, B) : A \cap B = \emptyset \}.
\end{eqnarray*}
The name of these graphs is derived from Kneser's conjecture
which states that $\chi(K(n,m)) = n - 2m + 2$.
Kneser's conjecture was proved by Lov\'{a}sz~\cite{Lov:Knesers}
in a famous paper which introduced the field of
``topological bounds'' on the chromatic number.

\subsection{Products, functors and Hedetniemi's conjecture}

Our examples so far have situated Hedetniemi's conjecture within
the more general problem of determining the chromatic number
of graphs of the type $R(K_n)$, where $R$ is some right adjoint
in the thin category of graphs. These questions can be easy or hard,
but there seems to be nothing special about the fact that $\times$ 
is the ``correct'' product in the category of graphs. The special
property of $\times$ that will be relevant to our discussion is 
the fact that thin right adjoints commute with it, 
up to homomorphic equivalence:
$$
R(G \times H) \leftrightarrow R(G) \times R(H).
$$
Here, $\times$ cannot be replaced by the various ``products''
of graph theory. It is a property that can be used to
say something meaningful about Hedetniemi's conjecture,
when suitable functors are found, as shown in Section~\ref{ag}.

\subsection{A right adjoint of a central Pultr functor}
\label{ss:racpf}

Our last example of this section introduces a functor
that we will use in Section~\ref{pf}.
It is our first example of a thin right adjoint that is not a central
Pultr functor. It can be motivated by the following question.
\begin{quote}
Does there exist, for every integer $n$, a graph $G_n$
such that $\chi(G_n) = n$ and $G_n$ admits a $n$-colouring
$f: G_n \rightarrow K_n$ where the neighbourhood of each 
colour class is an independent set?
\end{quote}
In \cite{GyaJenSti:On-graphs-with}, Gy\'{a}rf\'{a}s, Jensen and Stiebitz
present the question as a strengthening 
of a question of Harvey and Murty. Note that if two adjacent vertices 
$u_1$ and $u_2$ are respective neighbours of two identically coloured
vertices $u_0$ and $u_3$, then $u_0, u_1, u_2, u_3$ is a walk of length
three between two identically coloured vertices. Therefore the above
question can be reformulated in terms of the functor $\Gamma_{\T(3)}$
introduced in the first example of this section.
\begin{quote}
Does there exist, for every integer $n$, a graph $G_n$
such that $\chi(G_n) = \chi(\Gamma_{\T(3)}(G_n)) = n$?
\end{quote}
It turns out that $\Gamma_{\T(3)}$ has a thin right adjoint $\Omega_{\T(3)}$,
introduced in~\cite{Tar:Mul}.
For a graph $G$,  $\Omega_{\T_3}(G)$ is the graph constructed as follows.
The vertices of $\Omega_{\T_3}(G)$ are the couples $(u,U)$ such that
$u \in V(H)$, $U \subseteq V(G)$, and every vertex in $U$ is adjacent to $u$.
Two couples $(u,U)$, $(v,V)$ are joined by an edge of $\Omega_{\T_3}(H)$ 
if $u \in V$, $v \in U$, and every vertex
in $U$ is adjacent to every vertex in $V$.
It is easy to see that $\chi(\Omega_{\T(3)}(K_n)) \leq n$.
And since the condition $\Gamma_{\T(3)}(G_n) \rightarrow K_n$ is 
equivalent to $G_n \rightarrow \Omega_{\T(3)}(K_n)$, a graph
$G_n$ exists with the required properties if and only if
$G_n = \Omega_{\T(3)}(K_n)$ fits the bill, that is
$\chi(\Omega_{\T(3)}(K_n)) = n$. Once again, the question
reduces to determining the chromatic number of some
thin right adjoint of $K_n$. The question is answered in the
affirmative in \cite{GyaJenSti:On-graphs-with}.

\subsection{Excursion: Topological bounds on chromatic numbers}
As an aside note, we mention that the method for finding a lower bound
on $\chi(\Omega_{\T(3)}(K_n))$ in \cite{GyaJenSti:On-graphs-with}
is the topological method devised by Lov\'{a}sz \cite{Lov:Knesers}
for Kneser graphs. Thus the ``topological bounds'' are tight
for chromatic numbers of $\Omega_{\T(3)}(K_n)$ and 
$\Gamma_{\T(\circ, K_m)}(K_n)$. They are also trivially tight
for $\Gamma_{\T(\Box, H)}(K_n)$, and Hedetniemi's conjecture
would imply that they are also tight for 
$\Gamma_{\T(\times, H)}(K_n)$. Since it is easy to devise graphs
for which such topological bounds are not tight, it is intriguing
to see it work well in the case of many right adjoints.

Is it a coincidence? In modern terms, 
the bounds can be defined via functors which associate to a graph a 
``hom-complex'' and eventually an object in a category of topological 
spaces. This point of view is developed in Section 7 of \cite{DocSch:Topology}.
Perhaps the effectiveness of the topological bounds can be traced functorially
in some way. As far as we know, there is no unifying theory as to 
when topological bounds on the chromatic number are tight.

\section{The arc graph construction} \label{ag}

\subsection{Multiplicative complete graphs in $\mathcal{G}$ and $\mathcal{D}$,
and the Poljak-R\"odl function}

The complete graphs $K_1$ and $K_2$ can be shown to be
multiplicative both in $\mathcal{G}$ and $\mathcal{D}$
with relatively straightforward arguments. The first nontrivial
case of Hedetniemi's conjecture was established by El-Zahar 
and Sauer in the aptly named paper~\cite{EZaSau:Chro}. 

\begin{theorem}[\cite{EZaSau:Chro}] \label{k3mul}
$K_3$ is multiplicative in $\mathcal{G}$.
\end{theorem}

However, a few years before, 
Poljak and R\"odl~\cite{PolRod:Arc} had proved that for every
$n \geq 3$, there are digraphs $G_n$ and $H_n$ both with 
$n+1$ vertices and chromatic number $n+1$, such that
$\chi(G_n\times H_n) = n$. Thus $K_n$ fails to be multiplicative
in $\mathcal{D}$ for any $n \geq 3$. This justifies some
skepticism towards Hedetniemi's conjecture.  

Let $\phi: \mathbb{N} \rightarrow \mathbb{N}$ be defined by
$$
\phi(n) = \min \{ \chi(G\times H) : 
G, H \in \mathcal{G} \mbox{ and } \chi(G), \chi(H) = n\}.
$$
Hedetniemi's Conjecture~\ref{hc} is that $\phi(n) = n$ for all $n \in \mathbb{N}$,
but so far this has been verified only for $n \leq 4$. The weak
Hedetniemi Conjecture~\ref{whc} is that $\phi$~is unbounded.
Poljak and R\"odl proved the following result.
\begin{theorem}[\cite{PolRod:Arc}] \label{bfupr}
If $\phi$ is bounded, then its least upper bound is at most $16$.
\end{theorem}
This result and its developments, along with their relationship
with adjoint functors, are the topic of this section.

\subsection{The arc graph construction and its chromatic properties}

Let $\vec{P}_i$ be the directed path with $i$ arcs, that is,
the digraph defined by 
$V(\vec{P}_i) = \{ 0, \ldots, i\}$ and
$A(\vec{P}_i) = \{ (0,1), \ldots, (i-1,i) \}$.
Consider the Pultr template $\T = (\vec{P}_1,\vec{P}_2,\epsilon_1,\epsilon_2)$,
where $\epsilon_1$ and $\epsilon_2$ map the arc of $\vec{P}_1$ respectively
to the first and the second arc of $\vec{P}_2$. Then 
$\Gamma_{\T}: \mathcal{D} \rightarrow \mathcal{D}$ is the
{\em arc graph construction} usually denoted $\delta$:
The vertices of $\delta(G) = \Gamma_{\T}(G)$ are the arcs of
$G$, and its arcs correspond to pairs of consecutive arcs of $G$.
We will use the following property of the arc graph construction.

\begin{prop}[
  \cite{HarEnt:Arc-colorings}] \label{bocd}
For any digraph $G$, the following holds.
\begin{itemize}
\item[(i)] If $\chi(\delta(G)) \leq n$, then $\chi(G) \leq 2^{n}$.
\item[(ii)] If $\displaystyle \chi(G) \leq {n \choose {\lfloor n/2 \rfloor}}$,
then $\chi(\delta(G)) \leq n$.
\end{itemize}
\end{prop}

\begin{proof} We present a proof using adjoint functors.
It turns out that $\delta$ has a thin right adjoint, which we will call~$\delta_R$.
The vertices of $\delta_R(G)$ are the ordered pairs
$(U,V)$ of sets of vertices of $G$ such that $(u,v) \in A(G)$
for all $u \in U$ and $v \in V$. The arcs of $\delta_R(G)$
are the ordered pairs $((U,V),(W,X))$ such that $V \cap W \neq \emptyset$.
Note that the vertices of $\delta_R(K_n)$ are the ordered pairs
$(U,V)$ of disjoint sets of vertices of $K_n$. 
The map $f: \delta_R(K_n) \rightarrow \delta_R(K_n)$ defined by
$f(U,V) = (U, \overline{U})$ (where $\overline{U}$ is the complement of $U$)
is an endomorphism, with image of size $2^n$. Therefore
$\chi(\delta_R(K_n)) \leq 2^n$. Using adjunction we then get
$$
\chi(\delta(G)) \leq n
\Rightarrow  \delta(G) \rightarrow K_n
\Rightarrow G \rightarrow \delta_R(K_n)
\Rightarrow \chi(G) \leq 2^n.
$$
Also, $\delta_R(K_n)$ contains a copy of $\displaystyle K_{{n \choose \lfloor n/2 \rfloor}}$
induced by the sets $(U, \overline{U})$ such that $|U| = \lfloor n/2 \rfloor$.
Hence
$$
\chi(G) \leq {n \choose \lfloor n/2 \rfloor}
\Rightarrow G \rightarrow K_{{n \choose \lfloor n/2 \rfloor}} \rightarrow \delta_R(K_n)
\Rightarrow \delta(G) \rightarrow K_n
\Rightarrow \chi(\delta(G)) \leq n.
$$
\end{proof}

Therefore we have $\chi(\delta(G)) \simeq \log(\chi(G))$ and similarly
$\chi(\delta^n(G)) \simeq \log^n(\chi(G))$ (where the exponents represent composition). 
These are essentially the
only known examples of right adjoints with non-trivial predictable
effects on the chromatic number of general graphs. 
Chromatic numbers of graphs obtained by applying arc-graph-like functors
to transitive tournaments have recently been studied in~\cite{AvaLuczRod:On-generalized,AvaKayRei:The-chromatic}.
However, the following is not known.

\begin{problem} Suppose that $R: \mathcal{D} \rightarrow \mathcal{D}$ is 
a right adjoint such that for some unbounded functions 
$a, b: \mathbb{N} \rightarrow \mathbb{N}$ we have 
$a(\chi(G)) \leq \chi(R(G)) \leq b(\chi(G))$ for all $G \in \mathcal{D}$. 
Does it follow that for some $n$ we then have $\delta^n(G) \rightarrow R(G)$
for all $G \in \mathcal{D}$?
\end{problem} 

\subsection{The possible bounds on the Poljak-R\"odl function}

The directed version of Theorem~\ref{bfupr} is the following.
\begin{prop}\label{bfdpr}
Let $\psi: \mathbb{N} \rightarrow \mathbb{N}$ be defined by
$$
\psi(n) = \min \{ \chi(G\times H) : 
G, H \in \mathcal{D} \mbox{ and } \chi(G), \chi(H) = n\}.
$$
If $\psi$ is bounded, then its least upper bound is at most $4$.
\end{prop}
\begin{proof}
The proof uses the fact that $\delta$ is a right adjoint,
hence it commutes with the product. Let $b$ be the least upper bound
on~$\psi$. Let $n$ be an integer such that $\psi(n) = b$. Note that
$\psi$ is non-decreasing; thus there exist digraphs $G, H$ such that 
$\chi(G) = \chi(H) = 2^n$ and $\chi(G\times H) = b$. By Lemma~\ref{bocd},
we have $\chi(\delta(G)) \geq n$ and $\chi(\delta(H)) \geq n$,
whence $\chi(\delta(G) \times \delta(H)) \geq b$. However
$\delta(G) \times \delta(H) \leftrightarrow \delta(G\times H)$,
whence $\chi(\delta(G\times H)) \geq b$. By Lemma~\ref{bocd},
for any integer $m$ such that $b = \chi(G\times H) \leq {m \choose \lfloor m/2 \rfloor}$,
we have $b \leq \chi(\delta(G\times H)) \leq m$. The only integers $b\geq 3$
which satisfy this property are $3$ and $4$.
\end{proof}
The proof of Theorem~\ref{bfupr} is essentially an adaptation of the
proof of Corollary~\ref{bfdpr}. Indeed for $G, H \in \mathcal{G}$,
let $\vec{G}, \vec{H}$ be orientations of $G$ and $H$ respectively,
constructed by selecting one of the arcs $(u,v), (v,u)$ for each
edge $\{u,v\}$ of $G$ and $H$. Then $\chi(G\times H) = \chi(\vec{G} \times H)$,
since $\vec{G} \times H$ is an orientation of $G\times H$.
Furthermore the arcs of $\vec{G} \times H$ can be partitioned 
into arcs of $\vec{G} \times \vec{H}$ and arcs of $\vec{G} \times \vec{H}'$,
where $\vec{H}'$ is obtained by reversing the arcs of $\vec{H}$.
Theorem~\ref{bfupr} is then proved by adapting the argument of the 
proof of Corollary~\ref{bfdpr} to the function $\psi'$ defined by
$$
\psi'(n) = \min \{ \max\{ \chi(\vec{G}\times \vec{H}), \chi(\vec{G}\times \vec{H}') \}  : 
\vec{G}, \vec{H} \in \mathcal{D} \mbox{ and } \chi(\vec{G}), \chi(\vec{H}) = n\}.
$$
This accounts for the value $16 = 4^2$ in the statement of Theorem~\ref{bfupr}.
These results were improved shortly afterwards, independently by
Poljak, Schmerl and Zhu (the latter two unpublished):
\begin{theorem}[\cite{Pol:cdia}]
\begin{itemize}
\item[]
\item[(i)] If $\psi$ is bounded, then its least upper bound is 3.
\item[(ii)] If $\phi$ is bounded, then its least upper bound is at most 9.
\end{itemize}
\end{theorem}
The proof method is simply a finer analysis of the chromatic properties
of $\delta$; in particular this accounts for the value $9 = 3^2$
in the second statement. Thus the case of undirected graphs is never dealt 
with directly. Proof methods have always used directed graphs.
\begin{problem}
Is there a functor $R: \mathcal{G} \rightarrow \mathcal{G}$
that allows to provide tighter upper bounds on $\phi$,
if the case that it is bounded?
\end{problem}
A later result gives a further link between $\phi$ and $\psi$:
\begin{theorem}[\cite{TarWeh:prf}]
$\phi$ is bounded if and only if $\psi$ is bounded.
\end{theorem}
Thus the weak Hedetniemi Conjecture~\ref{whc} is equivalent for
graphs and for digraphs. It would be interesting to connect
the case of digraphs more directly to Hedetniemi's conjecture.
\begin{problem}
Is there a way to prove that if $\psi$ grows too slowly,
then Hedetniemi's conjecture is false?
(Note that if $\psi'$ is sub-quadratic, then Hedetniemi's conjecture
is false.)
\end{problem}

\subsection{Multiplicative complete graphs}

The last tentative application of~$\delta$ to Hedetniemi's conjecture
is an appealing argument that seems to have been first noticed
by Roman Ba\v{s}ic.

\begin{prop}\label{mcg}
If $K_n$ is multiplicative, then 
$\displaystyle K_{{n \choose \lfloor n/2 \rfloor}}$ is multiplicative.
\end{prop}

\begin{proof}
The proof uses a strengthening of Lemma~\ref{bocd} (ii) in the case
of undirected graphs: Sperner's theorem implies that if $G$ is an undirected
graph, then $\chi(\delta(G))$ is the largest integer $n$ such that
$\chi(G) \leq {n \choose \lfloor n/2 \rfloor}$. In other words,
for an undirected graph $G$ we have 
$$
\delta(G) \rightarrow K_n \Leftrightarrow G \rightarrow  K_{n \choose \lfloor n/2 \rfloor}.
$$
Now suppose that $G\times H \rightarrow K_{n \choose \lfloor n/2 \rfloor}$.
Then $\delta(G\times H) \rightarrow K_{n}$, that is,
$\delta(G) \times \delta(H) \rightarrow K_{n}$. If $K_n$ is multiplicative,
this implies $\delta(G) \rightarrow K_{n}$ or $\delta(H) \rightarrow K_{n}$,
whence $G\rightarrow K_{n \choose \lfloor n/2 \rfloor}$
or $H\rightarrow K_{n \choose \lfloor n/2 \rfloor}$. Thus if
$K_n$ is multiplicative, then 
$\displaystyle K_{{n \choose \lfloor n/2 \rfloor}}$ is multiplicative.
\end{proof}

The attentive reader may have noticed that the above argument switches between
$\mathcal{G}$ and $\mathcal{D}$. It correctly uses the multiplicativity
of $K_n$ in $\mathcal{D}$ to prove the multiplicativity of
$\displaystyle K_{{n \choose \lfloor n/2 \rfloor}}$ in $\mathcal{G}$.
However, multiplicativity is category sensitive. The precise formulation
of Proposition~\ref{mcg} is the following.

\begin{prop}[The precise formulation of Proposition~\ref{mcg}]\label{mcgprecise}
If $K_n$ is multiplicative in $\mathcal{D}$, then 
$\displaystyle K_{{n \choose \lfloor n/2 \rfloor}}$ is multiplicative
in $\mathcal{G}$.
\end{prop}

Thus the result turns out to be trivial, since the only complete graphs that are 
multiplicative in $\mathcal{D}$ are $K_0, K_1$ and $K_2$. 
Attempts have been made to fix the argument, or to find similar
arguments, using other functors.
(The symmetrisation of $\delta$ does not seem to be the correct functor, either.)
In Section~\ref{pf}, we examine the results obtained using
the functors related to the Pultr template $\T(3)$ of Section~\ref{af},
and its generalizations.

\subsection{Excursion: odd girth of shift graphs} 
Let $\vec{K}_n$ denote the transitive tournament on $n$ vertices,
that is,
\begin{eqnarray*}
V(\vec K_n) & = & \{0, 1, \ldots, n-1\}, \\
A(\vec K_n) & = & \{ (i,j) : i < j \}.
\end{eqnarray*}
The {\em shift graph} $S(n,k)$ is the graph $\delta^{k-1}(\vec{K}_n)$.
The shift graphs (or their symmetrisations) are folklore ``topology free''
examples of graphs with large chromatic number and no short odd cycles.
We show how these properties connect to adjoint functors.

\begin{prop} $S(n,k)$ has no odd cycle with fewer than $2k+1$ vertices,
and for $\log_2^{k-1}(n) > m$ we have $\chi(S(n,k)) \geq m$ (where the exponent
represents composition).
\end{prop}

\begin{proof}
We have $\chi(\vec{K}_n) = n$, hence the iterated logarithmic lower bounds on
$\chi(S(n,k))$ follow from Proposition~\ref{bocd}, that is, from the right
adjoint of $\delta$. We now show that the absence
of short odd cycles follows from the left adjoint of $\delta$.
Suppose that some orientation $C$ of an odd cycle admits a 
homomorphism to $S(n,k) = \delta(S(n,k-1))$, for some $k > 1$.
Then $\delta_L(C)$ admits a homomorphism to $S(n,k-1)$,
where $\delta_L$ is the left adjoint of $\delta$. By construction, 
the number of arcs in $\delta_L(C)$ is equal to the number of vertices
in $C$. Moreover, since $C$ admits a homomorphism to $S(n,k)$ which does not
contain an oriented cycle, $C$ must contain a source and a sink. The copies
of $\vec{P}_1$ corresponding to sources and sinks of $C$ are pendant;
therefore any odd cycle of $\delta_L(C)$ has fewer arcs than $C$.
Since we have $C \not \rightarrow \delta(K_2) \simeq K_2$, 
we get $\delta_L(C) \not \rightarrow K_2$ whence $\delta_L(C)$
indeed has an odd cycle. The smallest odd cycle in $\vec{K}_n$ has
three vertices; iteratively, this implies that the smallest
odd cycle in $S(n,k)$ has at least $2k+1$ vertices.
\end{proof}

\section{Path functors and circular graphs} \label{pf}
The natural generalization of Proposition~\ref{mcgprecise} is the following.
\begin{prop} \label{multomul}
Let $L: \mathcal{A} \rightarrow \mathcal{B}$ and
$R: \mathcal{B} \rightarrow \mathcal{A}$ be thin adjoint
functors such that $L(G \times H) \leftrightarrow L(G) \times L(H)$.
If $K$ is multiplicative in $\mathcal{B}$, then  
$R(K)$ is multiplicative in $\mathcal{A}$. Moreover,
if $L(R(G)) \leftrightarrow G$ for all $G \in \mathcal{B}$,
then the converse holds.
\end{prop}

The condition $L(G \times H) \leftrightarrow L(G) \times L(H)$ is satisfied whenever
$L$ is a central Pultr functor which admits a right adjoint. In this section,
we examine the applications of Proposition~\ref{multomul} to functors
associated to the template $\T(3)$ of Section~\ref{af} and its natural 
generalizations.

\subsection{The multiplicativity of odd cycles}
The functor $\Gamma_{\T(3)}: \mathcal{G} \rightarrow \mathcal{G}$ of Section~\ref{af}
admits the right adjoint $\Omega_{\T(3)}: \mathcal{G} \rightarrow \mathcal{G}$.
Therefore, using the fact that $K_3$ is multiplicative in $\mathcal{G}$,
we first get that $\Omega_{\T(3)}(K_3)$ is also multiplicative in $\mathcal{G}$. 
Then, recursively, the graphs $\Omega_{\T(3)}^n(K_3)$ are all multiplicative
in $\mathcal{G}$. It turns out that $\Omega_{\T(3)}^n(K_3)$ is homomorphically
equivalent to the odd cycle $C_{3^{n+1}}$ with $3^{n+1}$ vertices. So this sequence
of derivations yields particular cases of the following result of H\"{a}ggkvist, Hell,
Miller and Neumann Lara.
\begin{theorem}[\cite{HHMN:Mult}] \label{ocmul}
The odd cycles are all multiplicative in $\mathcal{G}$.
\end{theorem}
The proof given in \cite{HHMN:Mult} is an adaptation of the proof
of Theorem~\ref{k3mul} given in \cite{EZaSau:Chro}. 
However it is possible to derive the multiplicativity
of all odd cycles from that of $K_3$ using the following
adjoint functors that generalize
$\Gamma_{\T(3)}$ and $\Omega_{\T(3)}$:
Let $\T(2k+1) =(K_1,P_{2k+1},\epsilon_1,\epsilon_2)$, 
where $P_{2k+1}$ is the path with $2k+1$ edges and $\epsilon_1, \epsilon_2$ 
map $K_1$ to the endpoints of $P_{2k+1}$.
The right adjoint of $\Gamma_{\T(2k+1)}$ is $\Omega_{\T(2k+1)}$ is characterized
in \cite{Haj:On-colorings}. It is defined as follows.
\begin{itemize}
\item The vertices of $\Omega_{\T(2k+1)}(G)$ are the $(k+1)$-tuples 
  $(A_0, \ldots, A_k)$ such that each $A_i\subseteq V(G)$;
  $A_0$ is a singleton;
  every vertex of $A_{i-1}$ is adjacent to every vertex of $A_i$ for $i = 1, \ldots, k$;
  and $A_{i-1}$ is contained in $A_{i+1}$ for $i = 1, \ldots, k-1$.
\item The edges of $\Omega_{\T(2k+1)}(G)$ join pairs $(A_0, \ldots, A_k)$,
$(B_0, \ldots, B_k)$ such that $A_{i-1} \subseteq B_{i}$ and
$B_{i-1} \subseteq A_i$ for $i = 1, \ldots, k$, and every
vertex of $A_k$ is adjacent to every vertex of $B_k$.
\end{itemize}
To prove Theorem~\ref{ocmul} with these functors, we use the following.
\begin{lemma}[\cite{HaT:Graph-powers}] \label{compid}
$\Gamma_{\T(n)}(\Omega_{\T(n)}(G)) \leftrightarrow
G \leftrightarrow \Gamma_{\T(n)}(\Lambda_{\T(n)}(G))$
for any odd $n$ and $G \in \mathcal{G}$.
\end{lemma}

Theorem~\ref{ocmul} can then be proved as follows.
We note that $\Omega_{\T(2k+1)}(K_3)$ and $\Omega_{\T(3)}(C_{2k+1})$
are both isomorphic to the odd cycle $C_{3(2k+1)}$.
Therefore by Proposition~\ref{multomul}, the multiplicativity of $C_{3(2k+1)}$
follows from the multiplicativity of $K_3$, and the multiplicativity of 
$C_{2k+1}$ follows from the multiplicativity of $C_{3(2k+1)}$.

\subsection{Circular complete graphs and the circular chromatic number}

We now extend Theorem~\ref{ocmul} to some ``circular complete graphs'' defined as follows.
Let $\mathbb{Z}_s = \{0, 1, \ldots, s-1\}$ denote the cyclic group with $s$ elements.
For integers $r, s$ such that $1 \leq r \leq s/2$, the {\em circular complete graph}
$K_{s/r}$ is the graph with vertex set $\mathbb{Z}_s$ and with edges
$\{x, y\}$ such that $y-x \in \{r, r+1, \ldots, s-r\}$. In particular,
$K_{s/1}$ is the complete graph $K_s$, and $K_{(2r+1)/r}$ is the odd
cycle $C_{2r+1}$.

The conventional notation $K_{s/r}$ is slightly ambiguous, since it confuses
two integer parameters $s$ and $r$ with a single rational parameter $s/r$.
We will write $s/r$ for the parameter pair and $\frac{s}{r}$ for the corresponding
fraction. The notation is motivated by the following result,
which alleviates some of its ambiguity. 
\begin{lemma}[\cite{BoHe:star}] \label{phcg}
\begin{itemize}
\item[]
\item[(i)] $K_{s/r} \rightarrow K_{s'/r'}$ if and only if 
$\frac{s}{r} \leq \frac{s'}{r'}$.
\item[(ii)] For any $\frac{s}{r} > 2$, there exists 
$\frac{s'}{r'} < \frac{s}{r}$ such
that $K_{s/r}\setminus\{x\} \rightarrow K_{s'/r'}$
for all $x \in V(K_{s/r})$.
\end{itemize}
\end{lemma} 

The {\em circular chromatic number} $\chic(G)$ of a graph $G$ is defined as the
infimum of the values $\frac{s}{r}$ such that $G \rightarrow K_{s/r}$. By
Lemma \ref{phcg}, the infimum is attained. We have 
$\chic(G) \leq \chi(G) < \chic(G) + 1$, hence the circular chromatic
number is a refinement of the chromatic number.

As in the case of the chromatic number, the inequality 
$\chic(G\times H) \leq \min \{ \chic(G), \chic(H) \}$ follows
immediately from definitions. Zhu~\cite{Zhu:cirsur} conjectured that equality
always holds. This is a strengthening of Hedetniemi's conjecture.
It is equivalent to the statement that the circular complete graphs
are multiplicative. In this direction, the following is known.

\begin{prop}[\cite{Tar:Mul}] \label{moccg}
The circular complete graphs $K_{s/r}$ with $2 \leq \frac{s}{r} < 4$ are multiplicative.
\end{prop}

The proof uses the functors $\Gamma_{\mathcal{T}(3)}$ and $\Omega_{\mathcal{T}(3)}$
as follows.
\begin{itemize}
\item[(i)] When $\frac{s}{r} < 3$, we have 
$\Gamma_{\mathcal{T}(3)}(K_{s/r}) \leftrightarrow K_{s/(3r-s)}$.
\item[(ii)] When $\frac{s}{r} < \frac{12}{5}$, we have
$\Omega_{\mathcal{T}(3)}(\Gamma_{\mathcal{T}(3)}(K_{s/r})) \leftrightarrow K_{s/r}$.
\item[(iii)] Therefore by Proposition~\ref{multomul},
when $\frac{s}{r} < \frac{12}{5}$,
$K_{s/r}$ is multiplicative if and only if  $K_{s/(3r-s)}$ is multiplicative.
\item[(iv)] Starting with the odd cycles $C_{2n+1} = K_{(2n+1)/n}$ which 
are multiplicative by Proposition~\ref{ocmul}, we can use item (iii) above to infer
the multiplicativity of many circular complete graphs 
$K_{s/r} = \Gamma_{\mathcal{T}(3)}^m C_{2n+1}$. The set of rationals $\frac{s}{r}$
for which $K_{s/r}$ is proved to be multiplicative in this fashion is
a dense subset of the interval $(2,4)$. 
\item[(v)] The remaining graphs $K_{s/r}$, $2 < \frac{s}{r} < 4$
are proved to be multiplicative by a density argument: Suppose that
$$\frac{s'}{r'} = \chic\{G\times H\} < \min \{\chic(G), \chic(H)\} = \frac{s}{r}.$$
By item (iv) above there exists  $\frac{s'}{r'} < \frac{s''}{r''} < \frac{s}{r}$
such that $K_{s''/r''}$ is multiplicative. 
But we have $G\times H \rightarrow K_{s'/r'} \rightarrow K_{s''/r''}$
while $G, H \not \rightarrow K_{s''/r''}$, a contradiction.
\end{itemize}

\begin{corollary}
The identity $\chic(G\times H) = \min \{\chic(G), \chic(H)\}$
holds whenever $\min \{\chic(G), \chic(H)\} \leq 4$.
\end{corollary}
It turns out that the multiplicativity of $K_3$ implies that
of all the graphs $K_{s/r}: 2 < \frac{s}{r} < 4$. Up to homomorphic
equivalence, these are the only graphs known 
to be multiplicative (other than $K_1$ and $K_2$). 
The proof method exposed here does not extend to other
circular complete graphs. Property (i) breaks down at $\frac{s}{r} = 3$,
since $\Gamma_{\mathcal{T}(3)}(K_{s/r})$ contains loops whenever
$\frac{s}{r} \geq 3$. Property (ii) breaks down at $\frac{12}{5}$, we have
$\Gamma_{\mathcal{T}(3)}(K_{12/5}) \leftrightarrow K_{4}$, and
$\Omega_{\mathcal{T}(3)}(K_4)$ is $4$-chromatic by the result of
\cite{GyaJenSti:On-graphs-with} discussed in Section~\ref{af}.

\subsection{Compositions and chains of functors} For odd integers $m, n$,
we consider the functor 
$L^m_n = \Gamma_{\mathcal{T}(m)}\circ \Lambda_{\mathcal{T}(n)}$ 
and its right adjoint 
$R^n_m = \Gamma_{\mathcal{T}(n)}\circ \Omega_{\mathcal{T}(m)}$.
In particular, the circular complete graphs $K_{s/r}$ with
$s$ an odd multiple of $3$ are the graphs $R^m_n(K_3)$
($\simeq L^m_n(K_3)$) such that $m < 3n$. More generally,
for any odd cycle $C_k$ with $k \geq \frac{m}{n}$, 
we have $L^m_n(C_{2k+1}) \leftrightarrow R^m_n(C_{2k+1}) 
\leftrightarrow K_{s/r}$, where $s= nk$ and $r = \frac{nk-m}{2}$.

As noted in \cite{HaT:Graph-powers}, the 
circular chromatic number of $G$ can therefore be expressed 
in terms of the infimum of the values $\frac{m}{n}$ such that
$G \rightarrow R^m_n(K_3)$. Since $R^m_n$ is the thin right
adjoint of $L^n_m$, we can look instead at the values
$\frac{m}{n}$ such that $L^n_m(G)$ admits a homomorphism to $K_3$. 
In this characterization, the ``base graph'' $K_3$ could 
be replaced by any of the graphs $K_{s/r}$ with $2 < \frac{s}{r} < 4$,
with an appropriate change of parameters. However, replacing
$K_3$ by $K_4$ or other graphs yields new graph invariants,
defined by homomorphisms into chains of graphs which mimic
the circular complete graphs.

At the functorial level,
we can order thin functors in an obvious way, putting $F \leq F'$
if $F(G) \rightarrow F'(G)$ for every $G \in \mathcal{G}$.
\begin{prop} \label{pcaf}
$L^m_n \leq L^{m'}_{n'}$ if and only if $\frac{m}{n} \leq \frac{m'}{n'}$ and
similarly $R^m_n \leq R^{m'}_{n'}$ if and only if $\frac{m}{n} \leq \frac{m'}{n'}$.
\end{prop}
\begin{proof}
The implications $\frac{m}{n} \leq \frac{m'}{n'} \Rightarrow L^m_n \leq L^{m'}_{n'}$ 
and $R^m_n \leq R^{m'}_{n'}$
are proved in~\cite{HaT:Graph-powers}. The converse implications are witnessed by
suitably chosen circular complete graphs.
\end{proof}

Up to equivalence, we can assimilate $L^m_n$ and $R^m_n$ to functors 
$L_{m/n}$ and $R_{m/n}$ corresponding to the rational parameter $\frac{m}{n}$.
These are two chains of functors isomorphic to the positive rationals
with odd numerator and denominator. This order is in fact isomorphic to the rationals,
but our specific labelling leaves holes. In particular, the adjoint functors $L_{1/2}$
and $R_{2/1}$ are not defined. We next see that this specific pair of holes 
can be filled by Pultr functors.

\subsection{A limit functor} Consider the Pultr template 
$\mathcal{T}(2) = (\overline{P}_1,\allowbreak {P_1 \Box P_2},\allowbreak \epsilon_1, \epsilon_2)$
where $P_1$ and $P_2$ are path with vertex-sets $\{0,1\}$ and $\{0,1, 2\}$
respectively, $\Box$ is the Cartesian product of section~\ref{af}, and
$(\epsilon_1(0), \epsilon_1(1), \epsilon_2(0), \epsilon_2(1)) = 
((0,0),(1,0),(1,2),(0,2))$. Having $\overline{P}_1$ rather than $P_1$ 
as first coordinate of the template makes a difference only for graphs
with no edges. In all other cases, we could replace $\overline{P}_1$
by $P_1$ and get an equivalent functor.
\begin{prop}
For all $G \in \mathcal{G}$, we have
$L^m_n(G) \rightarrow \Lambda_{\mathcal{T}(2)}(G)$ when $\frac{m}{n} < \frac{1}{2}$,
and $\Lambda_{\mathcal{T}(2)}(G) \rightarrow L^m_n(G)$ when $\frac{m}{n} > \frac{1}{2}$;
as well as $R^n_m(G) \rightarrow \Gamma_{\mathcal{T}(2)}(G)$ when $\frac{n}{m} < 2$,
and $\Gamma_{\mathcal{T}(2)}(G) \rightarrow R^n_m(G)$ when $\frac{n}{m} > 2$.
\end{prop}

We omit the proof, which is similar to many proofs involving the
functors associated with the templates $\mathcal{T}(2k+1)$.
The main point is that $\Lambda_{\mathcal{T}(2)}$ and 
$\Gamma_{\mathcal{T}(2)}$ fulfill the role
of $L_{1/2}$ and $R_{2/1}$. It is not clear which of the other rational 
holes in the chains $\{L_{m/n}\}_{m, n \mbox{ \small odd}}$ and 
$\{R_{n/m}\}_{n, m \mbox{ \small odd}}$ can be filled, and how. 
It can be shown that the irrational holes cannot be filled by 
finite constructions. Thus, a workable theory of convergence
for sequences of thin functors remains to be developed.

One interesting aspect of the template  $\mathcal{T}(2)$ is that
$\Gamma_{\mathcal{T}(2)}$ has a {\em partial} right adjoint.
Let $\Omega_{\mathcal{T}(2)}(G)$ be the graph defined as follows.
\begin{itemize}
\item The vertices of $\Omega_{\mathcal{T}(2)}(G)$ are the
ordered pairs $(A,B) \subseteq V(G)^2$ such that every
vertex of $A$ is adjacent to every vertex of $B$ (i.e.,
$A$ and $B$ are {\em completely joined}).
\item The edges of $\Omega_{\mathcal{T}(2)}(G)$
are the pairs $\{(A,B),(C,D)\}$ such that
$A$ and $C$ are completely joined, $B$ and $D$ are completely joined,
$A \cap D \neq \emptyset$, and  $B \cap C \neq \emptyset$.
\end{itemize}
\begin{prop}
\begin{itemize}
\item[]
\item[(i)] If $\Gamma_{\mathcal{T}(2)}(G) \rightarrow H$,
then $G \rightarrow \Omega_{\mathcal{T}(2)}(H)$;
\item[(ii)] if $G \rightarrow \Omega_{\mathcal{T}(2)}(K_3)$,
then $\Gamma_{\mathcal{T}(2)}(G) \rightarrow K_3$;
\item[(iii)] however, $K_6 \rightarrow \Gamma_{\mathcal{T}(2)}(\Omega_{\mathcal{T}(2)}(K_4))
\not \rightarrow K_4$, hence $\Omega_{\mathcal{T}(2)}$ is not a right 
adjoint of $\Gamma_{\mathcal{T}(2)}$.
\end{itemize}
\end{prop}

We again omit the standard proof for the sake of briefness.
We note that the result allows the derivation of the multiplicativity of 
$\Omega_{\mathcal{T}(2)}(K_3) = K_{12/5}$
directly from that of $K_3$, instead of through density arguments as in
the proof of Proposition~\ref{moccg}. The range of graphs 
on which $\Omega_{\mathcal{T}(2)}$ acts as a right adjoint of 
$\Gamma_{\mathcal{T}(2)}$ is not precisely known, but it extends
to the odd cycles and allows to reprove the multiplicativity
of some circular complete graphs. 

In short, adjoint functors allow the extension of the proof of multiplicativity
of $K_3$ to the circular complete graphs $K_{s/r}, 2 < \frac{s}{r} < 4$,
sometimes in many ways, using compositions, limits and partial right adjoints.
However,
the multiplicativity of $K_4$, the ``next case of Hedetniemi's
conjecture'' remains open. We now present a difficulty in using the adjoint functors
method to derive the multiplicativity of other complete graphs from that of~$K_3$.

\subsection{Obstacle: A stronger form of multiplicativity}

El-Zahar and Sauer actually proved a stronger form of Theorem~\ref{k3mul}:
\begin{theorem}[\cite{EZaSau:Chro}] \label{k3strmul}
Let $G$ and $H$ be connected graphs containing odd cycles 
$G', H'$ respectively. Let $L$ be the subgraph of 
$G\times H$ induced by $V({G\times H'}) \cup V({G'\times H})$. 
If $L \rightarrow K_3$, then $G \rightarrow K_3$ or $H \rightarrow K_3$.
\end{theorem}
They conjectured that a similar strong form of multiplicativity
would hold for all complete graphs. However this conjecture
was refuted in~\cite{TarZhu:On-Hedetniemis}. The counterexamples
are of interest to us. Let $G_{m,n}$ be the graph obtained
by identifying one end of $P_n$ to a vertex of $K_{7/2}$
and the other end to a vertex of $K_m$. 
For $m >4$, $G_{m,n}\times G_{m,n} \leftrightarrow G_{m,n} \not \rightarrow K_4$.
Let $L_{m,n}$ be the subgraph of $G_{m,n} \times G_{m,n}$ induced
by $V(G_{m,n} \times K_{7/2}) \cup V(K_{7/2}\times G_{m,n})$.
\begin{theorem}[\cite{TarZhu:On-Hedetniemis}]
For $n \geq 3$ and $m$ arbitrary, $L_{m,n} \rightarrow K_4$.
\end{theorem}
Now suppose that for some $\mathcal{T}$, $\Gamma_{\mathcal{T}}$ admits a right adjoint
$\Omega_{\mathcal{T}}$ such that $\Omega_{\mathcal{T}}(K_3) = K_4$.
Then for $n \geq 3, m \geq 5$, we have $G_{m,n} \not \rightarrow K_4$ whence
$\Gamma_{\mathcal{T}}(G_{m,n}) \not \rightarrow K_3$.
However from $L_{m,n} \rightarrow K_4 = \Omega_{\mathcal{T}}(K_3)$, we get
$\Gamma_{\mathcal{T}}(L_{m,n}) \rightarrow K_3$. 
However $\Gamma_{\mathcal{T}}(L_{m,n})$ contains
a copy of $\Gamma_{\mathcal{T}}(G_{m,n}) \times \Gamma_{\mathcal{T}}(K_{7/2}) 
\cup \Gamma_{\mathcal{T}}(K_{7/2})\times \Gamma_{\mathcal{T}}(G_{m,n})$.
Hence by Theorem~\ref{k3strmul}, since 
$\Gamma_{\mathcal{T}}(G_{m,n}) \not \rightarrow K_3$,
we must have that $\Gamma_{\mathcal{T}}(K_{7/2})$ is bipartite. 
It can be shown that this is incompatible with the fact
that  $\Gamma_{\mathcal{T}}$ admits a right adjoint.

In short, the right adjoints we know seem to transfer a property
that is stronger than multiplicativity, and this seems to be
an obstacle in deriving the multiplicativity of other complete
graphs from that of~$K_3$.

We note that Wrochna~\cite{wrochna} has recently proved that all
square-free graphs are multiplicative.
The functors detailed here can therefore be used to derive from
this the multiplicativity of a larger class of graphs.
However, Wrochna states that his method also yields the stronger form of
multiplicativity for square-free graphs.
Hence the multiplicativity of larger complete graphs still seems out of reach.

\section{Concluding comments} \label{cc}
We have shown how some important facts concerning the weak Hedetniemi
conjecture and the characterization of multiplicative graphs are connected
to adjoint functors in the thin category of graphs. In addition,
the multi-factor version of the weak Hedetniemi conjecture
can also be tackled through adjoint functors, 
as shown in~\cite{FonNesTar:Interlacing}.
The characterization of central Pultr functors which admit right adjoints
is initiated in \cite{FonTar:Right-adjoints}. However, as we have seen,
further developments may depend on functors that escape the mould of Pultr
functors. In short, the tools devised so far in this line of study 
are categorial in nature, and seem to have reached their limit in their
present form. Perhaps the contribution of category theory to Hedetniemi's
conjecture would be to see the way to sharpen these tools.

\providecommand{\bysame}{\leavevmode\hbox to3em{\hrulefill}\thinspace}
\providecommand{\MR}{\relax\ifhmode\unskip\space\fi MR }
\providecommand{\MRhref}[2]{%
  \href{http://www.ams.org/mathscinet-getitem?mr=#1}{#2}
}
\providecommand{\href}[2]{#2}

\end{document}